\documentclass[reqno,a4paper]{amsart}
\usepackage{amssymb}%
\usepackage[hmarginratio=1:1]{geometry}%
\usepackage{ifpdf}
\usepackage{listings}
\usepackage{color,xcolor}
\usepackage{colortbl}
\usepackage{multirow}
\usepackage[numbers,sort&compress]{natbib}
\usepackage{booktabs}
\usepackage{float}
\ifpdf
 \usepackage[hyperindex]{hyperref}
\else
 \expandafter\ifx\csname dvipdfm\endcsname\relax
 \usepackage[hypertex,hyperindex]{hyperref}
 \else
 \usepackage[dvipdfm,hyperindex]{hyperref}
 \fi
\fi
\allowdisplaybreaks[4]
\theoremstyle{plain}
\newtheorem{thm}{Theorem}[section]

\newtheorem{cor}{Corollary}[section]

\newtheorem{mr}{Monotonicity rule}[section]
\theoremstyle{remark}
\newtheorem{rem}{Remark}[section]
\theoremstyle{plain}
\newtheorem{lem}{Lemma}[section]

\theoremstyle{definition}

\numberwithin{equation}{section}

\begin{document}

\title[]{Monotonicity rules for the ratio of two function series and two integral transforms}
\author{Zhong-Xuan Mao, Jing-Feng Tian*}

\address{Zhong-Xuan Mao \\
Hebei Key Laboratory of Physics and Energy Technology\\
Department of Mathematics and Physics\\
North China Electric Power University \\
Yonghua Street 619, 071003, Baoding, P. R. China}
\email{maozhongxuan\symbol{64}ncepu.edu.cn}

\address{Jing-Feng Tian\\
Hebei Key Laboratory of Physics and Energy Technology\\
Department of Mathematics and Physics\\
North China Electric Power University\\
Yonghua Street 619, 071003, Baoding, P. R. China}
\email{tianjf\symbol{64}ncepu.edu.cn}

\begin{abstract}
In this paper, we investigate the monotonicity of the functions $t \mapsto \frac{\sum_{k=0}^\infty a_k w_k(t)}{\sum_{k=0}^\infty b_k w_k(t)}$ and $x \mapsto \frac{\int_\alpha^\beta f(t) w(t,x) \textrm{d} t}{\int_\alpha^\beta g(t) w(t,x) \textrm{d} t}$, focusing on case where the monotonicity of $a_k/b_k$ and $f(t)/g(t)$ change once.
The results presented also provide insights into the monotonicity of the ratios of two power series, two $\mathcal{Z}$-transforms, two discrete Laplace transforms, two discrete Mellin transforms, two Laplace transforms, and two Mellin transforms. Finally, we employ these monotonicity rules to present several applications in the realm of special functions and stochastic orders.
\end{abstract}


\thanks{*Corresponding author: Jing-Feng Tian, e-mail addresses,
tianjf@ncepu.edu.cn}
\maketitle

\section{Introduction}
Let
\begin{equation} \label{S}
\mathcal{S}\{a_k\}(t) = \sum_{k=0}^\infty a_k w_k(t)
\end{equation}
be a function series and
\begin{equation} \label{T}
\mathcal{T}\{f(t)\}(x)=\int_\alpha^\beta f(t) w(t,x) \textrm{d} t
\end{equation}
be an integral transform that maps a function from its original function space to another function space through integration, where $w_k(t)$ and $w(t,x)$ are called the kernel functions.

Many extremely important and widely used series and transforms are special cases of (\ref{S}) and (\ref{T}). For example, important power series, Taylor series, $\mathcal{Z}$-transforms, discrete Laplace transforms
$\sum_{k=0}^\infty a_k e^{-kt}$, and discrete Mellin transforms $\sum_{k=0}^\infty a_k (k+1)^{t-1}$ all share the form of (\ref{S}); Fourier transforms, Fourier sine transforms, Fourier cosine transforms, Laplace transforms, Mellin transforms, Wavelet transforms, Hilbert transform, Laguerre transforms, and Laguerre transforms all exemplify particular instances of (\ref{T}).

The aforementioned series and transforms not only have a significant role in solving differential equations \cite{Nikiforov-1988,Bohner-PAMS-2017}, studying special functions \cite{Ismail-CJM-1977,Ismail-AP-1977,Ismail-SIAM-JMA-1979,Ismail-PAMS-1982}, and exploring number theory and complex analysis in the realm of mathematics, but they are also widely used in diverse fields such as engineering, physics, signal processing, image processing, computer vision, telecommunications, geophysics, finance, and control theory.

The primary focus of this article is to establish the monotonicity rules for the ratios of
two function series
\begin{equation} \label{AB}
t \mapsto \frac{\mathcal{A}(t):=\sum_{k=0}^\infty a_k w_k(t)}{\mathcal{B}(t):=\sum_{k=0}^\infty b_k w_k(t)},
\end{equation}
and two integral transforms
\begin{equation} \label{FG}
x \mapsto \frac{\mathcal{F}(x):=\int_\alpha^\beta f(t) w(t,x) \textrm{d} t}{\mathcal{G}(x):=\int_\alpha^\beta g(t) w(t,x) \textrm{d} t}.
\end{equation}
We consider the case that the monotonicity of $\{a_k/b_k\}$ and $t\mapsto f(t)/g(t)$ change once.

This paper will also present some applications of these monotonicity rules. First, we provide some applications in the field of special functions: we establish the monotonicity of the functions
\begin{equation} \label{R1}
t \mapsto \frac{\mathcal{DE}_{a,b}(t) := \sum_{k=0}^\infty \frac{w_k(t)}{\Gamma(a k+b)}  }{\mathcal{DE}_{c,d}(t):=\sum_{k=0}^\infty \frac{ w_k(t)}{\Gamma(c k+d)}},
\end{equation}
and
\begin{equation} \label{R2}
x \mapsto \frac{\mathcal{CE}_{a,b}(t) := \int_0^\infty \frac{w(t,x)}{\Gamma(a t+b)}  \textrm{d} t }{\mathcal{CE}_{c,d}(t) := \int_0^\infty \frac{w(t,x)}{\Gamma(a t+b)} \textrm{d} t}.
\end{equation}
Subsequently, we introduce applications within the realm of stochastic orders, specifically, we offer comparisons between two discrete and continuous random variables by utilizing the Laplace transform ratio order.

\section{Main results}

In this section, we study the monotonicity of (\ref{AB}) and (\ref{FG}). Firstly, we introduce several sets of sequence of functions $\{w_k(t)\}_{k\geq0}$ and binary functions $w(t,x)$.

Let $\mathcal{DW}_{11}(r)(\mathcal{DW}_{12}(r))$ be the set of all sequences of functions $\{w_k(t)\}_{k\geq0}$ that satisfy the following conditions:
\begin{itemize}
\item[(i)] $w_k(t)$ is positive and $w_k^\prime(t)$ is continuous on $(0,r)$ for all $k\geq0$,
\item[(ii)] the sequence $k\mapsto w_k^\prime(t)/ w_k(t)$ is increasing (decreasing) for all fixed $t\in(0,r)$.
\end{itemize}
Let $\mathcal{DW}_{2}(r)$ be the set of all sequences of functions $\{w_k(t)\}_{k\geq0}$ that satisfy the following conditions:
\begin{itemize}
\item[(i)] $w_k(t)$ is positive as well as $w_k^\prime(t)$ and $w_k^{\prime\prime}(t)$ are continuous on $(0,r)$ for all $k\geq0$,
\item[(ii)] $\frac{w_k(0^+)}{w_0(0^+)} = 0$ for all $k\geq 1$ with $w_0(0^+) = c \in (0,\infty)$, and $\frac{w_k^\prime(0^+)}{w_1^\prime(0^+)} = 0$ for all $k\geq 2$,
\item[(iii)] $w_0^\prime(t)=0$  and $w_k^\prime(t)\geq0 (k=1,2,\cdots)$ for all $t\in (0,r)$,
\item[(iv)] the sequence $k \mapsto w_k^{\prime\prime}(t)/ w_k^{\prime}(t)$ is increasing for all fixed $t\in(0,r)$.
\end{itemize}
Let $\mathcal{DW}_{3}(r)$ be the set of all sequences of functions $\{w_k(t)\}_{k\geq0}$ that satisfy the following conditions:
\begin{itemize}
\item[(i)] $w_k(t)$ is positive as well as $w_k^\prime(t)$ and $w_k^{\prime\prime}(t)$ are continuous on $(0,r)$ for all $k\geq0$,
\item[(ii)] $\frac{w_k(r^-)}{w_0(r^-)} = 0$ for all $k\geq 1$ with $w_0(r^-) = c \in (0,\infty)$, and $\frac{w_k^\prime(r^-)}{w_1^\prime(r^-)} = 0$ for all $k\geq 2$,
\item[(iii)] $w_0^\prime(t)=0$ and $w_k^\prime(t)\leq0 (k=1,2,\cdots)$ for all $t\in (0,r)$,
\item[(iv)] the sequence $k \mapsto w_k^{\prime\prime}(t)/ w_k^{\prime}(t)$ is decreasing for all fixed $t\in(0,r)$.
\end{itemize}

\begin{rem}
Clearly, $\{t^k\}_{k\geq0} \in \mathcal{DW}_{11}(r) \cap \mathcal{DW}_{2}(r)$, $\{t^{-k}\}_{k\geq0} \in \mathcal{DW}_{12}(r) \cap \mathcal{DW}_{3}(\infty)$, $\{e^{-kt}\}_{k\geq0} \in \mathcal{DW}_{12}(r) \cap \mathcal{DW}_{3}(\infty)$, $\{(k+1)^{-t}\}_{k\geq0} \in \mathcal{DW}_{12}(r) \cap \mathcal{DW}_{3}(\infty)$.
\end{rem}

Likewise, let $\mathcal{CW}_{11}(\alpha,\beta)(\mathcal{CW}_{12}(\alpha,\beta))$ be the set of all binary functions $w(t,x)$ that satisfy the following conditions:
\begin{itemize}
\item[(i)] $w(t,x)$ is positive as well as $w(t,x)$ and $w_x(t,x)$ are continuous on $[\alpha,\beta] \times (0,\infty)$,
\item[(ii)] the function $t\mapsto w_x(t,x)/w(t,x)$ is increasing (decreasing) on $[\alpha,\beta]$ for all fixed $x\in(0,\infty)$.
\end{itemize}
Let $\mathcal{CW}_{2}(\alpha,\beta)$ be the set of all binary functions $w(t,x)$ that satisfy the following conditions:
\begin{itemize}
\item[(i)] $w(t,x)$ is positive as well as $w(t,x)$, $w_x(t,x)$, and $w_{xx}(t,x)$ are continuous on $[\alpha,\beta] \times (0,\infty)$,
\item[(ii)] $\frac{w(t,0^+)}{w(\alpha,0^+)} = 0$ for all $t\in(\alpha,\beta]$ with $w(\alpha,0^+) = c \in (0,\infty)$, and $\frac{w_x(t,0^+)}{w_x(\alpha,0^+)} = 0$ for all $t\in(\alpha,\beta]$,
\item[(iii)] $w_x(\alpha,x)=0$ for all $x\in (0,\infty)$, and $w_x(t,x)\geq 0$ for all $(t,x)\in (\alpha,\beta] \times (0,\infty)$,
\item[(iv)] the function $t\mapsto w_{xx}(t,x)/ w_x(t,x)$ is increasing on $[\alpha,\beta]$ for all fixed $x\in(0,\infty)$.
\end{itemize}
Let $\mathcal{CW}_{3}(\alpha,\beta)$ be the set of all binary functions $w(t,x)$ that satisfy the following conditions:
\begin{itemize}
\item[(i)] $w(t,x)$ is positive as well as $w(t,x)$, $w_x(t,x)$, and $w_{xx}(t,x)$ are continuous on $[\alpha,\beta] \times (0,\infty)$,
\item[(ii)] $\frac{w(t,\infty)}{w(\alpha,\infty)} = 0$ for all $t\in(\alpha,\beta]$ with $w(\alpha,\infty) = c \in (0,\infty)$, and $\frac{w_x(t,\infty)}{w_x(\alpha,\infty)} = 0$ for all $t\in(\alpha,\beta]$,
\item[(iii)] $w_x(\alpha,x)=0$ for all $x\in (0,\infty)$, and $w_x(t,x)\leq 0$ for all $(t,x)\in (\alpha,\beta] \times (0,\infty)$,
\item[(iv)] the function $t\mapsto w_{xx}(t,x)/ w_x(t,x)$ is decreasing on $[\alpha,\beta]$ for all fixed $x\in(0,\infty)$.
\end{itemize}

\begin{rem}
Clearly, $x^t \in \mathcal{CW}_{11}(\alpha,\beta) \cap \mathcal{CW}_{2}(\alpha,\beta)$,
$x^{-t} \in \mathcal{CW}_{12}(\alpha,\beta) \cap \mathcal{CW}_{3}(\alpha,\beta)$,
$e^{-tx} \in \mathcal{CW}_{12}(\alpha,\beta) \cap \mathcal{CW}_{3}(\alpha,\beta)$,
$(t+1)^{-x} \in \mathcal{CW}_{12}(\alpha,\beta) \cap \mathcal{CW}_{3}(\alpha,\beta)$.
\end{rem}

\begin{rem}
For convenience, we denote that $\mathcal{DW}_{1}(r):=\mathcal{DW}_{11}(r) \cup \mathcal{DW}_{11}(r)$ and $\mathcal{CW}_{1}(\alpha,\beta):=\mathcal{CW}_{11}(\alpha,\beta) \cup \mathcal{CW}_{12}(\alpha,\beta)$.
\end{rem}

Let $-\infty \leq a < b\leq \infty$, functions $F$ and $G$ be differentiable on $(a,b)$, and $G\neq0$ on $(a,b)$. Then the Yang's $H$ function \cite{Yang-JMAA-2015} is defined by
\begin{equation*}
H_{F,G} =\frac{F^\prime}{G^\prime} G -F.
\end{equation*}
Clearly
\begin{equation*}
\Big( \frac{F}{G} \Big)^\prime = \frac{G^\prime}{G^2} H_{F,G},
\end{equation*}
and if $F$ and $G$ are twice differentiable, then
\begin{equation*}
H_{F,G}^\prime= \Big( \frac{F^\prime}{G^\prime} \Big)^\prime G.
\end{equation*}

\subsection{Monotonicity rules for the ratio of two function series}

In this section, we consider the monotonicity of function  (\ref{AB}). Initially, the case that the sequence $\{a_k/b_k\}$ is either increasing or decreasing for all $k\geq 0$ had been given in \cite[Lemma 2.2]{Koumandos-MS-2009}.

\begin{mr} \label{mr1-0}
(\cite[Lemma 2.2]{Koumandos-MS-2009}) Let the sequence of functions $\{w_k(t)\}_{k\geq0} \in \mathcal{DW}_{1}(r)$, both $\sum_{k=0}^\infty a_k w_k^{(l)}(t)$ and $\sum_{k=0}^\infty b_k w_k^{(l)}(t)$ converge uniformly on $(0,r)$ for $l=0,1$ and $a_k,b_k\geq 0$ for all $k$. If $\{a_k/b_k\}$ is increasing (decreasing) for all $k \geq 0$ and $\{w_k(t)\}_{k\geq0} \in \mathcal{DW}_{11}(r)$, then the function $t\mapsto \mathcal{A}(t)/\mathcal{B}(t)$ is increasing (decreasing) on $(0,r)$; if $\{a_k/b_k\}$ is increasing (decreasing) for all $k \geq 0$ and $\{w_k(t)\}_{k\geq0} \in \mathcal{DW}_{12}(r)$, then the function $t\mapsto \mathcal{A}(t)/\mathcal{B}(t)$ is decreasing (increasing) on $(0,r)$.
\end{mr}

\begin{rem}
In Monotonicity rule \ref{mr1-0}, $\sum_{k=0}^\infty a_k w_k^{(l)}(t)$ is regarded as
\begin{equation*}
\sum_{k=0}^\infty a_k w_k^{(l)}(t) =
\begin{cases}
\sum_{k=0}^\infty a_k w_k(t), & l=0, \\
\sum_{k=0}^\infty a_k w_k^\prime(t), & l=1.
\end{cases}
\end{equation*}
\end{rem}


Now we consider the case that there exists an integer $m\geq 1$ such that the sequence $\{a_k/b_k\}$ is increasing (decreasing) for all $0 \leq k \leq m-1$ and decreasing (increasing) for all $k \geq m$.

\begin{mr} \label{mr1-1}
Let the sequence of functions $\{w_k(t)\}_{k\geq0} \in \mathcal{DW}_{2}(r)$, both $\sum_{k=0}^\infty a_k w_k^{(l)}(t)$ and $\sum_{k=0}^\infty b_k w_k^{(l)}(t)$ converge uniformly on $(0,r)$ for $l=0,1,2$ and $a_k,b_k\geq 0$ for all $k$. Suppose that the sequence $\{a_k/b_k\}$ is increasing (decreasing) for all $0\leq k\leq m$ and decreasing (increasing) for all $k\geq m$. Then
\begin{itemize}
\item[(i)] the function $t\mapsto \mathcal{A}(t)/\mathcal{B}(t)$ is increasing (decreasing) on $(0,r)$ if and only if $H_{\mathcal{A},\mathcal{B}}(r^-)\geq$ ($\leq$) $0$;
\item[(ii)] there exists $t_0\in (0,r)$ such that the function $t\mapsto \mathcal{A}(t)/\mathcal{B}(t)$ is increasing (decreasing) on $(0, t_0)$ and  decreasing (increasing) on $(t_0,r)$ if $H_{\mathcal{A},\mathcal{B}}(r^-)< (>)0$.
\end{itemize}
\end{mr}

\begin{proof}
Without loss of generality, we consider the case that the sequence $\{a_k/b_k\}$ is increasing for $0\leq k \leq m$ and decreasing for $k\geq m$. Clearly, from $H_{\mathcal{A},\mathcal{B}}=(\mathcal{A}/\mathcal{B})^\prime \mathcal{B}^2/\mathcal{B}^\prime \geq 0$, we conclude the necessity in the first statement. Next, we prove the sufficiency by using mathematical induction.

(I) When $m=1$, that is, the sequence $\{a_k/b_k\}$ is increasing for $k=0,1$ and decreasing for all $k\geq 1$. Form the formulas
\begin{equation*}
\mathcal{A}^\prime(t)=\sum_{k=0}^\infty a_k w_k^\prime(t) = \sum_{k=1}^\infty a_k w_k^\prime(t), \quad
\mathcal{B}^\prime(t)=\sum_{k=0}^\infty b_k w_k^\prime(t) = \sum_{k=1}^\infty b_k w_k^\prime(t),
\end{equation*}
and using Monotonicity rule \ref{mr1-0}, we obtain that the function $t\mapsto\mathcal{A}^\prime(t)/\mathcal{B}^\prime(t)$ is decreasing on $(0,r)$.
Since $H_{\mathcal{A},\mathcal{B}}^\prime = (\mathcal{A}^\prime/\mathcal{B}^\prime)^\prime \mathcal{B} \leq 0$, the function $t\mapsto H_{\mathcal{A},\mathcal{B}}(t)$ is decreasing on $(0,r)$.

Noting that $H_{\mathcal{A},\mathcal{B}}(0^+)=b_0 \big( \frac{a_1}{b_1} - \frac{a_0}{b_0} \big) \geq 0$, we divide the proof into two cases.

(i) If $H_{\mathcal{A},\mathcal{B}}(r^-) \geq 0$, then $H_{\mathcal{A},\mathcal{B}}\geq0$ and $(\mathcal{A}/\mathcal{B})^\prime = \mathcal{B}^{-2}\mathcal{B}^\prime H_{\mathcal{A},\mathcal{B}} \geq 0$.

(ii) If $H_{\mathcal{A},\mathcal{B}}(r^-) < 0$, then there exists $t_0 \in (0,r)$ such that $H_{\mathcal{A},\mathcal{B}}(t) \geq 0$ for all $t \in (0, t_0)$ and $H_{\mathcal{A},\mathcal{B}}(t) \leq 0$ for all $t \in (t_0,r)$. Therefore, from $(\mathcal{A}/\mathcal{B})^\prime = \mathcal{B}^{-2}\mathcal{B}^\prime H_{\mathcal{A},\mathcal{B}}$, the function $t\mapsto \mathcal{A}(t)/\mathcal{B}(t)$ is first increasing and then decreasing.

(II) Suppose it is true for $m=n$, that is, the sequence $\{a_k/b_k\}$ is increasing for $0\leq k\leq n$ and decreasing for all $k\geq n$. Now we consider the case $m=n+1$. Since
\begin{equation*}
\frac{\mathcal{A}^\prime(t)}{\mathcal{B}^\prime(t)} = \frac{\sum_{k=1}^\infty a_k w_k^\prime(t)}{\sum_{k=1}^\infty b_k w_k^\prime(t)} = \frac{\sum_{k=0}^\infty a_{k+1} w_{k+1}^\prime(t)}{\sum_{k=0}^\infty b_{k+1} w_{k+1}^\prime(t)}
\end{equation*}
and the sequence $\{a_{k+1}/b_{k+1}\}$ is increasing for all $0\leq k \leq n$ and decreasing for all $k\geq n$. Thus, by using induction hypothesis, the function $t\mapsto\mathcal{A}^\prime(t)/\mathcal{B}^\prime(t)$ is increasing on $(0,r)$ if $H_{\mathcal{A}^\prime,\mathcal{B}^\prime}(r^-) \geq 0$, and there exists $t_1 \in (0,r)$ such that the function $t\mapsto\mathcal{A}^\prime(t)/\mathcal{B}^\prime(t)$ is increasing on $(0,t_1)$ and decreasing on $(t_1,r)$ if $H_{\mathcal{A}^\prime,\mathcal{B}^\prime}(r^-) < 0$.

(i) If $H_{\mathcal{A}^\prime,\mathcal{B}^\prime}(r^-) \geq 0$, then $H^\prime_{\mathcal{A},\mathcal{B}}=(\mathcal{A}^\prime/\mathcal{B}^\prime)^\prime \mathcal{B} \geq 0$, $H_{\mathcal{A},\mathcal{B}} \geq 0$, and $(\mathcal{A}/\mathcal{B})^\prime = \mathcal{B}^{-2} \mathcal{B}^\prime H_{\mathcal{A},\mathcal{B}} \geq 0$.

(ii) If $H_{\mathcal{A},\mathcal{B}}(r^-) \geq0$ and $H_{\mathcal{A}^\prime,\mathcal{B}^\prime}(r^-) <0$, then both $\mathcal{A}^\prime/\mathcal{B}^\prime$ and $H_{\mathcal{A},\mathcal{B}}$ are first increasing and then decreasing. Thus $H_{\mathcal{A},\mathcal{B}}\geq0$ and $(\mathcal{A}/\mathcal{B})^\prime = \mathcal{B}^{-2} \mathcal{B}^\prime H_{\mathcal{A},\mathcal{B}} \geq 0$.

(iii) If $H_{\mathcal{A},\mathcal{B}}(r^-) <0$ and $H_{\mathcal{A}^\prime,\mathcal{B}^\prime}(r^-) <0$, then both $\mathcal{A}^\prime/\mathcal{B}^\prime$ and $H_{\mathcal{A},\mathcal{B}}$ are first increasing and then decreasing. Thus there exists $t_1 \in (0,r)$ such that $H_{\mathcal{A},\mathcal{B}}(t) \geq 0$ for all $t\in(0,t_1)$ and $H_{\mathcal{A},\mathcal{B}}(t) \leq 0$ for all $t\in(t_1,\infty)$.
\end{proof}

Now we consider the case that $\{w_k(t)\}_{k\geq0} \in \mathcal{DW}_{3}(r)$, which leads to the following Monotonicity rule.

\begin{mr} \label{mr1-2}
Let $\{w_k(t)\}_{k\geq0} \in \mathcal{DW}_{3}(r)$, both $\sum_{k=0}^\infty a_k w_k^{(l)}(t)$ and $\sum_{k=0}^\infty a_k,b_k w_k^{(l)}(t)$ converge uniformly on $(0,r)$ for $l=0,1,2$ and $a_k, b_k\geq 0$ for all $k$. Suppose that the sequence $\{a_k/b_k\}$ is increasing (decreasing) for all $0\leq k\leq m$ and decreasing (increasing) for all $k\geq m$. Then
\begin{itemize}
\item[(i)] the function $t\mapsto \mathcal{A}(t)/\mathcal{B}(t)$ is decreasing (increasing) on $(0,r)$ if and only if $H_{\mathcal{A},\mathcal{B}}(0^+)\geq$ ($\leq$) $0$;
\item[(ii)] there exists $t_0\in (0,r)$ such that the function $t\mapsto \mathcal{A}(t)/\mathcal{B}(t)$ is increasing (decreasing) on $(0, t_0)$ and decreasing (increasing) on $(t_0,r)$ if $H_{\mathcal{A},\mathcal{B}}(0^+)< (>)0$.
\end{itemize}
\end{mr}

\begin{proof}
The proof is similar to that of Monotonicity rules \ref{mr1-1}, hence it is omitted here.
\end{proof}

\begin{rem}
Monotonicity rules \ref{mr1-0}, \ref{mr1-1}, and \ref{mr1-2} also hold for the case of $r\to \infty$. Moreover, let $a_k=b_k=0$ for all $k \geq N$, these three rules also investigate the monotonicity of the function
\begin{equation*}
t \mapsto \frac{\sum_{k=0}^N a_k w_k(t) }{\sum_{k=0}^N b_k w_k(t)}.
\end{equation*}
\end{rem}

\subsubsection{The monotonicity rules for the ratio of two power series}

Taking $w_k(t)=t^k$, then the functions $\mathcal{A}$ and $\mathcal{B}$ reduce to the so-called power series.
\begin{equation*}
\mathcal{A}(t)= \sum_{k=0}^\infty a_k t^{k}, \quad \mathcal{B}(t)=\sum_{k=0}^\infty b_k t^{k}.
\end{equation*}
In this case, Monotonicity rule \ref{mr1-0} reduces to \cite[Lemma 1]{Biernacki-AUMCS-1955} (see also \cite[Lemma 2.1]{Ponnusamy-M-1997}).
\begin{cor} \label{cor-pow-1}
(\cite[Lemma 1]{Biernacki-AUMCS-1955}) Let $\mathcal{A}(t)=\sum_{k=0}^\infty a_k t^{k}$ and $\mathcal{B}(t)=\sum_{k=0}^\infty b_k t^{k}$ be two real power series converging on $(-r,r)$ and $b_k\geq 0$ for all $k$. If $\{a_k/b_k\}$ is increasing (decreasing) for all $k \geq 0$, then the function $t\mapsto \mathcal{A}(t)/\mathcal{B}(t)$ is increasing (decreasing) on $(0,r)$.
\end{cor}

Moreover, Monotonicity rule \ref{mr1-1} reduces to \cite[Theorem 2.1]{Yang-JMAA-2015}.
\begin{cor} \label{cor-pow-2}
(\cite[Theorem 2.1]{Yang-JMAA-2015}) Let $\mathcal{A}(t)=\sum_{k=0}^\infty a_k t^{k}$ and $\mathcal{B}(t)=\sum_{k=0}^\infty b_k t^{k}$ be two real power series converging on $(-r,r)$ and $b_k\geq 0$ for all $k$. Suppose that the sequence $\{a_k/b_k\}$ is increasing (decreasing) for all $0\leq k\leq m$ and decreasing (increasing) for all $k\geq m$. Then
\begin{itemize}
\item[(i)] the function $t\mapsto \mathcal{A}(t)/\mathcal{B}(t)$ is increasing (decreasing) on $(0,r)$ if and only if $H_{\mathcal{A},\mathcal{B}}(r^-)\geq$ ($\leq$) $0$;
\item[(ii)] there exists $t_0\in (0,r)$ such that the function $t\mapsto \mathcal{A}(t)/\mathcal{B}(t)$ is increasing (decreasing) on $(0, t_0)$ and decreasing (increasing) on $(t_0,r)$ if $H_{\mathcal{A},\mathcal{B}}(r^-)< (>)0$.
\end{itemize}
\end{cor}

\begin{rem}
The case that $r\to \infty$ can be referred to \cite[Corollary 2.3]{Yang-JMAA-2015}.
\end{rem}

The Taylor series, having the form
\begin{equation*}
f(t) = \sum_{k=0}^\infty \frac{f^{(k)}(0)}{k!} t^k
\end{equation*}
is considered a special kind of power series. Taking $a_k=\frac{f^{(k)}(0)}{k!}$ and $b_k=\frac{g^{(k)}(0)}{k!}$, Corollaries \ref{cor-pow-1} and \ref{cor-pow-2} reduce to the following two corollaries, respectively.

\begin{cor}
Let functions $f(t)$ and $g(t)$ have arbitrary order derivatives at point $t=0$,
as well as $f(t)=\sum_{k=0}^\infty \frac{f^{(k)}(0)}{k!} t^{k}$ and $g(t)=\sum_{k=0}^\infty \frac{g^{(k)}(0)}{k!} t^{k}$ be converged on $(-r,r)$ with $g^{(k)}(0)\geq0$ for all $k\geq0$.
If $\{f^{(k)}(0)/g^{(k)}(0)\}$ is increasing (decreasing) for all $k \geq 0$, then the function $t\mapsto f(t)/g(t)$ is increasing (decreasing) on $(0,r)$.
\end{cor}

\begin{cor}
Let $f(t)=\sum_{k=0}^\infty \frac{f^{(k)}(0)}{k!} t^{k}$ and $g(t)=\sum_{k=0}^\infty \frac{g^{(k)}(0)}{k!} t^{k}$ be two Taylor series converging on $(-r,r)$ and $b_k\geq 0 $ for all $k$. Suppose that the sequence $\{f^{(k)}(0)/g^{(k)}(0)\}$ is increasing (decreasing) for all $0\leq k\leq m$ and decreasing (increasing) for all $k\geq m$. Then
\begin{itemize}
\item[(i)] the function $t\mapsto f(t)/g(t)$ is increasing (decreasing) on $(0,r)$ if and only if $H_{f,g}(r^-)\geq (\leq) 0$;
\item[(ii)] there exists $t_0\in (0,r)$ such that the function $t\mapsto f(t)/g(t)$ is increasing (decreasing) on $(0, t_0)$ and  decreasing (increasing) on $(t_0,r)$ if $H_{f,g}(r^-)< (>)0$.
\end{itemize}
\end{cor}

\subsubsection{The monotonicity rules for the ratio of two $\mathcal{Z}$-transforms}

Taking $w_k(t)=t^{-k}$, then the functions $\mathcal{A}$ and $\mathcal{B}$ reduce to the so-called $\mathcal{Z}$-transforms
\begin{equation*}
\mathcal{A}(t)= \sum_{k=0}^\infty a_k t^{-k}, \quad \mathcal{B}(t)=\sum_{k=0}^\infty b_k t^{-k}.
\end{equation*}
Monotonicity rules \ref{mr1-0} and \ref{mr1-2} respectively reduce to Corollaries \ref{cor-Ztran-1} and \ref{cor-Ztran-2} by taking $r\to\infty$, which establish the monotonicity of the ratio of two $\mathcal{Z}$-transforms.

\begin{cor} \label{cor-Ztran-1}
Let $\mathcal{A}(t)=\sum_{k=0}^\infty a_k t^{-k}$ and $\mathcal{B}(t)=\sum_{k=0}^\infty b_k t^{-k}$ converge uniformly on $(0,\infty)$ and $b_k\geq 0$ for all $k$. If $\{a_k/b_k\}$ is increasing (decreasing) for all $k \geq 0$, then the function $t\mapsto \mathcal{A}(t)/\mathcal{B}(t)$ is decreasing (increasing) on $(0,\infty)$.
\end{cor}

\begin{cor} \label{cor-Ztran-2}
Let $\mathcal{A}(t)=\sum_{k=0}^\infty a_k t^{-k}$ and $\mathcal{B}(t)=\sum_{k=0}^\infty b_k t^{-k}$ converge uniformly on $(0,\infty)$ and $b_k\geq 0$ for all $k$.
If $\{a_k/b_k\}$ is increasing (decreasing) for all $0\leq k\leq m$ and decreasing (increasing) for all $k\geq m$, then there exists $t_0\in (0,\infty)$ such that the function $t\mapsto \mathcal{A}(t)/\mathcal{B}(t)$ is increasing (decreasing) on $(0,t_0)$ and decreasing (increasing) on $(t_0,\infty)$.
\end{cor}

\begin{proof}
By using Monotonicity rule \ref{mr1-2}, it suffices to check that $H_{A,B}(0^+)<(>)0$. Note that
\begin{equation*}
H_{A,B}(0^+)= H_{A(1/t),B(1/t)}(\infty).
\end{equation*}
the fact that the right-hand side of the equation is negative (positive) was proved in the proof of \cite[Corollary 2.3]{Yang-JMAA-2015}.
\end{proof}

\subsubsection{The monotonicity rules for the ratio of two discrete Laplace transforms}

Taking $w_k(t)=e^{-k t}$, then the function $\mathcal{A}$ and $\mathcal{B}$ reduce to
\begin{equation*}
\mathcal{A}(t)= \sum_{k=0}^\infty a_k e^{-k t}, \quad \mathcal{B}(t)=\sum_{k=0}^\infty b_k e^{-k t},
\end{equation*}
which are considered as the discrete Laplace transforms in this paper.

Monotonicity rules \ref{mr1-0} and \ref{mr1-2} respectively reduce to Corollaries \ref{cor-DLap-1} and \ref{cor-DLap-2} by taking $r\to\infty$, which establish the monotonicity of the ratio of two discrete Laplace transforms.

\begin{cor} \label{cor-DLap-1}
Let $\mathcal{A}(t)=\sum_{k=0}^\infty a_k e^{-k t}$ and $\mathcal{B}(t)=\sum_{k=0}^\infty b_k e^{-k t}$ converge on $(0,\infty)$ and $a_k,b_k\geq 0$ for all $k$. If $\{a_k/b_k\}$ is increasing (decreasing) for all $k \geq 0$, then the function $t\mapsto \mathcal{A}(t)/\mathcal{B}(t)$ is decreasing (increasing) on $(0,\infty)$.
\end{cor}

\begin{cor} \label{cor-DLap-2}
Let $\mathcal{A}(t)=\sum_{k=0}^\infty a_k e^{-k t}$ and $\mathcal{B}(t)=\sum_{k=0}^\infty b_k e^{-k t}$ converge on $(0,\infty)$ and $a_k, b_k\geq 0$ for all $k$.
Suppose that the sequence $\{a_k/b_k\}$ is increasing (decreasing) for all $0\leq k\leq m$ and decreasing (increasing) for all $k\geq m$. Then
\begin{itemize}
\item[(i)] the function $t\mapsto \mathcal{A}(t)/\mathcal{B}(t)$ is decreasing (increasing) on $(0,r)$ if and only if $H_{\mathcal{A},\mathcal{B}}(0^+)\geq (\leq) 0$;
\item[(ii)] there exists $t_0\in (0,\infty)$ such that the function $t\mapsto \mathcal{A}(t)/\mathcal{B}(t)$ is increasing (decreasing) on $(0, t_0)$ and decreasing (increasing) on $(t_0,\infty)$ if $H_{\mathcal{A},\mathcal{B}}(0^+)< (>)0$.
\end{itemize}
\end{cor}

\subsubsection{The monotonicity rules for the ratio of two discrete Mellin transforms}
Taking $w_k(t)=(k+1)^{-t}$, then the functions $\mathcal{A}$ and $\mathcal{B}$ reduce to
\begin{equation*}
\mathcal{A}(t)= \sum_{k=0}^\infty a_k (k+1)^{-t}, \quad \mathcal{B}(t)=\sum_{k=0}^\infty b_k (k+1)^{-t}.
\end{equation*}
Clearly, $\mathcal{A}(1-t)$ can be regarded as discrete Mellin transform, $\mathcal{A}(t)$ is the Riemann zeta function when $a_k \equiv 1$, and $\mathcal{A}(-m)$ is $m$-th order moment of a random variable with distribution $P(X=k)=a_{k-1}, k=1,\cdots$.

Monotonicity rules \ref{mr1-0} and \ref{mr1-2} respectively reduce to Corollaries \ref{cor-DMel-1} and \ref{cor-DMel-2} by taking $r\to\infty$, which establish the monotonicity of the ratio of two discrete Mellin transforms.

\begin{cor} \label{cor-DMel-1}
Let $\mathcal{A}(t)=\sum_{k=0}^\infty a_k (k+1)^{-t}$ and $\mathcal{B}(t)=\sum_{k=0}^\infty b_k (k+1)^{-t}$ converge on $(0,\infty)$ and $a_k, b_k\geq 0$ for all $k$. If $\{a_k/b_k\}$ is increasing (decreasing) for all $k \geq 0$, then the function $t\mapsto \mathcal{A}(t)/\mathcal{B}(t)$ is decreasing (increasing) on $(0,\infty)$ and the function $t\mapsto \mathcal{A}(1-t)/\mathcal{B}(1-t)$ is increasing (decreasing) on $(-\infty,1)$.
\end{cor}

\begin{cor} \label{cor-DMel-2}
Let $\mathcal{A}(t)=\sum_{k=0}^\infty a_k (k+1)^{-t}$ and $\mathcal{B}(t)=\sum_{k=0}^\infty b_k (k+1)^{-t}$ converge on $(0,\infty)$ and $a_k,b_k\geq 0$ for all $k$.
Suppose that $\{a_k/b_k\}$ is increasing (decreasing) for all $0\leq k\leq m$ and decreasing (increasing) for all $k\geq m$. Then
\begin{itemize}
\item[(i)] the function $t\mapsto \mathcal{A}(t)/\mathcal{B}(t)$ is decreasing (increasing) on $(0,\infty)$ and the function $t \mapsto \mathcal{A}(1-t)/\mathcal{B}(1-t)$ is increasing (decreasing) on $(-\infty,1)$ if and only if $H_{\mathcal{A},\mathcal{B}}(0^+)\geq$ ($\leq$) $0$;
\item[(ii)] there exists $t_0\in (0,\infty)$ such that the function $t\mapsto \mathcal{A}(t)/\mathcal{B}(t)$ is increasing (decreasing) on $(0, t_0)$ and decreasing (increasing) on $(t_0,\infty)$ and the function $t \mapsto \mathcal{A}(1-t)/\mathcal{B}(1-t)$ is increasing (decreasing) on $(-\infty,1-t_0)$ and decreasing (increasing) on $(1-t_0, 1)$ if $H_{\mathcal{A},\mathcal{B}}(0^+)< (>)0$.
\end{itemize}
\end{cor}

\subsection{Monotonicity rules for the ratio of two integral transforms}

In this section, we consider the monotonicity of the function (\ref{FG}).

The case that the function $t\mapsto f(t)/g(t)$ consistently maintains monotonic increases or decreases on $(0,\infty)$ has already been considered in \cite[Lemma 9]{Qi-CRM-2022}.

\begin{mr} \label{mr2-0}
(\cite[Lemma 9]{Qi-CRM-2022}) Let $w(t,x)\in\mathcal{CW}_1(\alpha,\beta)$ and functions $f(t)$, $g(t)>0$ be integrable in $t\in(\alpha,\beta)$. If $t\mapsto f(t)/g(t)$ is increasing (decreasing) on $(\alpha,\beta)$, and $w(t,x)\in\mathcal{CW}_{11}(\alpha,\beta)$, then the function (\ref{FG}) is increasing (decreasing) on $(0,\infty)$; if $t\mapsto f(t)/g(t)$ is increasing (decreasing) on $(\alpha,\beta)$, and $w(t,x)\in\mathcal{CW}_{12}(\alpha,\beta)$, then the function (\ref{FG}) is decreasing (increasing) on $(0,\infty)$.
\end{mr}

Now we consider the case that there exists $t^*\in(\alpha,\beta)$ such that the function $t\mapsto f(t)/g(t)$ is increasing (decreasing) on $[\alpha,t^*]$ and decreasing (increasing) on $[t^*,\beta]$.

\begin{mr} \label{mr2-1}
Let $0 < \alpha < \beta < \infty$, $w(t,x)\in\mathcal{CW}_2(\alpha,\beta)$, the functions $\mathcal{F}(x)=\int_{\alpha}^{\beta} f(t) w(t,x) \textrm{d} t$ and $\mathcal{G}(x)=\int_{\alpha}^{\beta} g(t) w(t,x) \textrm{d} t$ be defined on $(0,\infty)$, where $f,g>0$ are continuous functions on $[\alpha,\beta]$. If there exists $t^*\in (\alpha,\beta)$ such that the function $t\mapsto f(t)/g(t)$ is increasing (decreasing) on $[\alpha,t^*]$ and decreasing (increasing) on $[t^*,\beta]$, then
\begin{itemize}
\item[(i)] the function $x\mapsto \mathcal{F}(x)/\mathcal{G}(x)$ is increasing (decreasing) on $(0,\infty)$ if and only if $H_{\mathcal{F},\mathcal{G}}(\infty) \geq (\leq)0$;
\item[(ii)] there exists $x^*>0$ such that the function $x\mapsto \mathcal{F}(x)/\mathcal{G}(x)$ is increasing (decreasing) on $(0,x^*)$ and decreasing (increasing) on $(x^*,\infty)$ if $H_{\mathcal{F},\mathcal{G}}(\infty) < (>)0$.
\end{itemize}
\end{mr}

\begin{proof}
Without loss of generality, we consider the case that the function $t \mapsto f(t)/g(t)$ is increasing on $[\alpha,t^*]$ and decreasing on $[t^*,\beta]$. The necessity of the first statement is obvious.

Let $m\geq 1$ be an integer, $\{t_i\}_{i\geq0}$ be an increasing sequence with $t_0 =\alpha$, $t_m =t^*$, $\lim_{n\to \infty} t_n=\beta$, and $\sup_{i\geq0} (t_{i+1}-t_{i}) =0$, as well as
\begin{equation*}
\mathcal{F}_i(x) := f(t_i) w(t_i,x) \Delta t_i,
\quad
\mathcal{G}_i(x) := g(t_i) w(t_i,x) \Delta t_i.
\end{equation*}
By the definition of definite integral, we have
\begin{equation*}
\mathcal{F}(x) = \sum_{i=0}^{\infty} \mathcal{F}_i(x) = \sum_{i=0}^{\infty} f(t_i) w(t_i,x) \Delta t_i , \quad \mathcal{G}(x) = \sum_{i=0}^{\infty} \mathcal{G}_i(x)  = \sum_{i=0}^{\infty} g(t_i) w(t_i,x) \Delta t_i.
\end{equation*}

Since the function $t\mapsto f(t)/g(t)$ is increasing on $[t_0,t_m]$ and decreasing on $[t_m,t_\infty]$, then the sequence $f(t_i)/g(t_i)$ is increasing for all $0\leq i \leq m$ and decreasing for all $i \geq m$.

By using Monotonicity rule \ref{mr1-1} and the fact that $w(t,x)\in\mathcal{CW}_2(\alpha,\beta)$, we obtain that the function $x\mapsto \mathcal{F}(x)/\mathcal{G}(x)$ is increasing on $(0,\infty)$ if $H_{\mathcal{F},\mathcal{G}}(\infty)\geq0$ and there exists $x^*>0$ such that the function $x\mapsto \mathcal{F}(x)/\mathcal{G}(x)$ is increasing on $(0,x^*)$ and decreasing on $(x^*,\infty)$ if $H_{\mathcal{F},\mathcal{G}}(\infty) < 0$.
\end{proof}

\begin{rem}
Moreover, we have
\begin{equation*}
\lim_{x \to 0} \frac{\mathcal{F}(x)}{\mathcal{G}(x)} = \frac{\int_{\alpha}^\beta f(t) w(t,0^+) \textrm{d} t}{\int_{\alpha}^\beta g(t) w(t,0^+) \textrm{d} t} = \frac{\int_{\alpha}^\beta f(t) \textrm{d} t}{\int_{\alpha}^\beta g(t) \textrm{d} t}.
\end{equation*}
\end{rem}

Likewise, we have the following monotonicity rule.
\begin{mr} \label{mr2-2}
Let $0 <\alpha< \beta< \infty$, $w(t,x)\in\mathcal{CW}_3(\alpha,\beta)$, the functions $\mathcal{F}(x)=\int_{\alpha}^{\beta} f(t) w(t,x) \textrm{d} t$ and $\mathcal{G}(x)=\int_{\alpha}^{\beta} g(t) w(t,x) \textrm{d} t$ be defined on $(0,\infty)$, where $f,g>0$ are continuous functions on $[\alpha,\beta]$. If there exists $t^*\in (\alpha,\beta)$ such that the function $t\mapsto f(t)/g(t)$ is increasing (decreasing) on $[\alpha,t^*]$ and decreasing (increasing) on $[t^*,\beta]$, then
\begin{itemize}
\item[(i)] the function $x\mapsto \mathcal{F}(x)/\mathcal{G}(x)$ is decreasing (increasing) on $(0,\infty)$ if and only if $H_{\mathcal{F},\mathcal{G}}(0^+) \geq (\leq)0$;
\item[(ii)] there exists $x^*>0$ such that the function $x\mapsto \mathcal{F}(x)/\mathcal{G}(x)$ is increasing (decreasing) on $(0,x^*)$ and decreasing (increasing) on $(x^*,\infty)$ if $H_{\mathcal{F},\mathcal{G}}(0^+) < (>)0$.
\end{itemize}
\end{mr}

\begin{rem}
Moreover, we have
\begin{equation*}
\lim_{x \to \infty} \frac{\mathcal{F}(x)}{\mathcal{G}(x)} = \frac{\int_{\alpha}^\beta f(t) w(t,\infty) \textrm{d} t}{\int_{\alpha}^\beta g(t) w(t,\infty) \textrm{d} t} = \frac{\int_{\alpha}^\beta f(t) \textrm{d} t}{\int_{\alpha}^\beta g(t) \textrm{d} t}.
\end{equation*}
\end{rem}

Let $\alpha \to 0$ and $\beta \to \infty$. Then we obtain the following two corollaries.
\begin{cor}
Let $w(t,x)\in\mathcal{CW}_2(\alpha,\beta)$, the functions $\mathcal{F}(x)=\int_{\alpha}^{\beta} f(t) \frac{\partial^l w(t,x)}{\partial x^l} \textrm{d} t$ and $\mathcal{G}(x)=\int_{\alpha}^{\beta} g(t)  \frac{\partial^l w(t,x)}{\partial x^l} \textrm{d} t$ converge uniformly on $(0,\infty)$ for all $l=0,1,2$, where $f,g>0$ are continuous functions on $(0,\infty)$. If there exists $t^*\in (0,\infty)$ such that the function $t\mapsto f(t)/g(t)$ is increasing (decreasing) on $(0,t^*]$ and decreasing (increasing) on $[t^*,\infty)$, then
\begin{itemize}
\item[(i)] the function $x\mapsto \mathcal{F}(x)/\mathcal{G}(x)$ is increasing (decreasing) on $(0,\infty)$ if and only if $H_{\mathcal{F},\mathcal{G}}(\infty) \geq (\leq)0$;
\item[(ii)] there exists $x^*>0$ such that the function $x\mapsto \mathcal{F}(x)/\mathcal{G}(x)$ is increasing (decreasing) on $(0,x^*)$ and decreasing (increasing) on $(x^*,\infty)$ if $H_{\mathcal{F},\mathcal{G}}(\infty) < (>)0$.
\end{itemize}
\end{cor}

\begin{cor}
Let $w(t,x)\in\mathcal{CW}_3(\alpha,\beta)$, the functions $\mathcal{F}(x)=\int_{\alpha}^{\beta} f(t) \frac{\partial^l w(t,x)}{\partial x^l} \textrm{d} t$ and $\mathcal{G}(x)=\int_{\alpha}^{\beta} g(t) \frac{\partial^l w(t,x)}{\partial x^l} \textrm{d} t$ converge uniformly on $(0,\infty)$ for all $l=0,1,2$, where $f,g>0$ are continuous functions on $(0,\infty)$. If there exists $t^*\in (0,\infty)$ such that the function $t\mapsto f(t)/g(t)$ is increasing (decreasing) on $(0,t^*]$ and decreasing (increasing) on $[t^*,\infty)$, then
\begin{itemize}
\item[(i)] the function $x\mapsto \mathcal{F}(x)/\mathcal{G}(x)$ is decreasing (increasing) on $(0,\infty)$ if and only if $H_{\mathcal{F},\mathcal{G}}(0^+) \geq (\leq)0$;
\item[(ii)] there exists $x^*>0$ such that the function $x\mapsto \mathcal{F}(x)/\mathcal{G}(x)$ is increasing (decreasing) on $(0,x^*)$ and decreasing (increasing) on $(x^*,\infty)$ if $H_{\mathcal{F},\mathcal{G}}(0^+) < (>)0$.
\end{itemize}
\end{cor}

\subsubsection{The monotonicity rules for the ratio of two Laplace transforms}
Taking $w(t,x)=e^{-tx}$ and $\alpha \to 0, \beta \to \infty$, then the functions $\mathcal{F}$ and $\mathcal{G}$ reduce to Laplace transforms
\begin{equation*}
\mathcal{F}(x)= \int_{0}^\infty f(t) e^{-tx} \textrm{d} t, \quad \mathcal{G}(x)= \int_{0}^\infty g(t) e^{-tx} \textrm{d} t,
\end{equation*}
respectively.
Monotonicity rules \ref{mr2-0} and \ref{mr2-2} respectively reduce to \cite[Lemma 4]{Yang-JIA-2017-317} and \cite[Theorem 2]{Yang-JMAA-2019}.

\begin{cor} \label{cor-lap-1}
(\cite[Lemma 4]{Yang-JIA-2017-317}) Let the functions $f,g$ be defined on $(0,\infty)$ such that their Laplace transforms exist with $g \neq 0$. If the function $t\mapsto f(t)/g(t)$ is increasing (decreasing) on $(0,\infty)$, then the function
\begin{equation*}
x \mapsto \frac{\mathcal{L}\{f(t)\}(x)}{\mathcal{L}\{g(t)\}(x)}= \frac{\int_0^\infty f(t) e^{-tx} \textrm{d} t }{\int_0^\infty g(t) e^{-tx} \textrm{d} t }
\end{equation*}
is decreasing (increasing) on $(0,\infty)$.
\end{cor}

\begin{cor} \label{cor-lap-2}
(\cite[Theorem 2]{Yang-JMAA-2019}) Let the Laplace transforms of the continuous functions $f, g>0$ exist and $\mathcal{F}(x):=\mathcal{L}\{f(t)\}(x), \mathcal{G}(x):=\mathcal{L}\{g(t)\}(x)$. If there exists $t^*\in (0,\infty)$ such that the function $t\mapsto f(t)/g(t)$ is increasing (decreasing) on $(0,t^*)$ and decreasing (increasing) on $(t^*,\infty)$, then
\begin{itemize}
\item[(i)] the function $x\mapsto \mathcal{F}(x)/\mathcal{G}(x)$ is decreasing (increasing) on $(0,\infty)$ if and only if $H_{\mathcal{F},\mathcal{G}}(0^+) \geq (\leq)0$;
\item[(ii)] there exists $x^*>0$ such that the function $x\mapsto \mathcal{F}(x)/\mathcal{G}(x)$ is increasing (decreasing) on $(0,x^*)$ and decreasing (increasing) on $(x^*,\infty)$ if $H_{\mathcal{F},\mathcal{G}}(0^+) < (>)0$.
\end{itemize}
\end{cor}

\subsubsection{The monotonicity rules for the ratio of two Mellin transforms}

Taking $w(t,x)=t^{x-1}$ and $\alpha \to 0, \beta \to \infty$, then the functions $\mathcal{F}$ and $\mathcal{G}$ reduce to Mellin transforms
\begin{equation*}
\mathcal{F}(x)= \mathcal{M}\{f(t)\}(x)= \int_{0}^\infty f(t) t^{x-1} \textrm{d} t, \quad \mathcal{G}(x)= \mathcal{M}\{g(t)\}x) = \int_{0}^\infty g(t) t^{x-1} \textrm{d} t,
\end{equation*}
respectively.
Monotonicity rule \ref{mr2-0} reduces to Corollary \ref{cor-Mel-1}.

\begin{cor} \label{cor-Mel-1}
Let the Mellin transforms of the functions $f, g>0$ exist. If the function $t\mapsto f(t)/g(t)$ is increasing (decreasing) on $(0,\infty)$, then the function
\begin{equation*}
x \mapsto \frac{\mathcal{M}\{f(t)\}(x)}{\mathcal{M}\{g(t)\}(x)} = \frac{\int_{0}^\infty f(t) t^{x-1} \textrm{d} t}{\int_{0}^\infty g(t) t^{x-1} \textrm{d} t}
\end{equation*}
is increasing (decreasing) on $(0,\infty)$.
\end{cor}

However, $w(t,x)=t^{x-1} \notin \mathcal{CW}_2(0,\infty) \cup \mathcal{CW}_3(0,\infty)$. Thus, we take $w(t,x)=(t+1)^{-x}$ and $\alpha \to 0, \beta \to \infty$, then the functions $\mathcal{F}$ and $\mathcal{G}$ reduce to
\begin{equation*}
\mathcal{F}(x)= \int_{0}^\infty f(t) (t+1)^{-x} \textrm{d} t, \quad \mathcal{G}(x) = \int_{0}^\infty g(t) (t+1)^{-x}\textrm{d} t,
\end{equation*}
respectively.
Monotonicity rule \ref{mr2-2} reduces to Corollary \ref{cor-Mel-2}.

\begin{cor} \label{cor-Mel-2}
Let $\mathcal{F}(x)= \int_{0}^\infty f(t) (t+1)^{-x} \textrm{d} t$ and $\mathcal{G}(x)= \int_{0}^\infty g(t) (t+1)^{-x} \textrm{d} t$ exist for all $x\in(0,\infty)$ with $g>0$. If there exists $t^*\in (0,\infty)$ such that $t\mapsto f(t)/g(t)$ is increasing (decreasing) on $(0,t^*)$ and decreasing (increasing) on $(t^*,\infty)$, then
\begin{itemize}
\item[(i)] the function $x\mapsto \mathcal{F}(x)/\mathcal{G}(x)$ is decreasing (increasing) on $(0,\infty)$ if and only if $H_{\mathcal{F},\mathcal{G}}(0^+) \geq (\leq)0$;
\item[(ii)] there exists $x^*>0$ such that the function $x\mapsto \mathcal{F}(x)/\mathcal{G}(x)$ is increasing (decreasing) on $(0,x^*)$ and decreasing (increasing) on $(x^*,\infty)$ if $H_{\mathcal{F},\mathcal{G}}(0^+) < (>)0$.
\end{itemize}
\end{cor}

\section{Applications}

\subsection{Applications in special functions}
For more applications of the previously given monotonicity rule in the field of special functions, please refer to \cite{Yang-PAMS-2017,Yang-RACSAM-2019,Yang-PAMS-2022,Yang-JMAA-2022,Tian-RM-2023,Mao-CRM-2023} and their references.

Let $D=(0,\infty)^4= \cup_{i=1}^7 D_i$, where $\Gamma(x) = \int_0^\infty e^{-t} t^{x-1} \textrm{d} t$ is the gamma function and $\psi(x)=\Gamma^\prime(x)/\Gamma(x)$ is the psi function, and
\begin{equation} \label{D}
\begin{split}
D_1 &:= \{ (a,b,c,d) : a=c, a,b,c,d>0 \}, \\
D_2 &:= \{ (a,b,c,d) : a>c, d \leq b, a,b,c,d>0 \}, \\
D_3 &:= \{ (a,b,c,d) : a>c, d > b, c \psi(d) - a \psi (b) \leq 0, a,b,c,d>0 \}, \\
D_4 &:= \{ (a,b,c,d) : a>c, d > b, c \psi(d) - a \psi (b) >  0, a,b,c,d>0 \}, \\
D_5 &:= \{ (a,b,c,d) : a<c, b \leq d, a,b,c,d>0 \}, \\
D_6 &:= \{ (a,b,c,d) : a<c, b > d, c \psi(d) - a \psi (b) \geq 0, a,b,c,d>0 \}, \\
D_7 &:= \{ (a,b,c,d) : a<c, b > d, c \psi(d) - a \psi (b) <  0, a,b,c,d>0 \}.
\end{split}
\end{equation}

First, we consider the monotonicity of $t \mapsto \Gamma(ct+d)/\Gamma(at+b)$ for all cases where $a,b,c,d>0$.
\begin{lem} \label{lem1}
Let $a, b, c, d>0$ and $D_i(i=1,2,\cdots,7)$ be defined by $(\ref{D})$.
\begin{itemize}
\item[(i)] If $(a,b,c,d)\in D_1$, then the function
    \begin{equation} \label{r-gam-2}
    t \mapsto \frac{\Gamma(ct+d)}{\Gamma(at+b)}
    \end{equation}
    is increasing (decreasing) on $(0,\infty)$ if and only if $d \geq (\leq) b$;
\item[(ii)] if $(a,b,c,d)\in D_2 \cup D_3$, then the function (\ref{r-gam-2}) is decreasing on $(0,\infty)$;
\item[(iii)] if $(a,b,c,d)\in D_4$, then there exist $t^*\in (0,\infty)$ such that the function (\ref{r-gam-2}) is increasing on $(0,t^*)$ and decreasing on $(t^*,\infty)$;
\item[(iv)] if $(a,b,c,d)\in D_5 \cup D_6$, then the function (\ref{r-gam-2}) is increasing on $(0,\infty)$;
\item[(v)] if $(a,b,c,d)\in D_7$, then there exist $t^*\in (0,\infty)$ such that the function (\ref{r-gam-2}) is decreasing on $(0,t^*)$ and increasing on $(t^*,\infty)$.
\end{itemize}
\end{lem}

\begin{proof}
Let
\begin{equation*}
Q_{a,b,c,d}(t):= c \psi( c t+d) - a \psi (at+b), \quad t>0.
\end{equation*}
Then
\begin{equation*}
\Big( \frac{\Gamma(ct+d)}{\Gamma(at+b)} \Big)^\prime = \frac{\Gamma(ct+d)}{\Gamma(at+b)} Q_{a,b,c,d}(t).
\end{equation*}

If $(a,b,c,d)\in D_1$, then
\begin{equation*}
Q_{a,b,c,d}(t) = a \psi( a t+d) - a \psi (at+b) \geq (\leq) 0
\end{equation*}
if and only if $d \geq (\leq) b$.

Without loss of generality, we consider the case that $a>c$.
If $d\leq b$, then
\begin{equation*}
Q_{a,b,c,d}(t) = c \psi( c t+d) - a \psi (at+b) \leq c \psi( c t+b) - a \psi (at+b) \leq 0.
\end{equation*}
If $d> b$, then
\begin{eqnarray*}
Q_{a,b,c,d}^\prime(t) &=&  c^2 \psi^\prime( c t+d) - a^2 \psi^\prime (at+b) \\
&=& \sum_{k=0}^\infty \Big( \frac{c^2}{(ct+d+k)^2} - \frac{a^2}{(at+b+k)^2} \Big) \\
&=& \sum_{k=0}^\infty \frac{A k^2 + Bk +C}{(at+b+k)^2(ct+d+k)^2} < 0,
\end{eqnarray*}
where
\begin{eqnarray*}
A &:=& -(a^2-c^2) <0,\\
B &:=& 2\big( bc^2-a^2d+ac^2t-a^2c t \big) <0,\\
C &:=&  b^2c^2-a^2d^2+2abc^2t-2a^2cdt <0.
\end{eqnarray*}
With the facts that
\begin{equation*}
\lim_{t\to \infty} Q_{a,b,c,d}(t) < 0,
\end{equation*}
and
\begin{equation*}
\lim_{t\to 0} Q_{a,b,c,d}(t) = c \psi(d) - a\psi(b),
\end{equation*}
we obtain the desired conclusion.
\end{proof}

Clearly, the following lemma can be deduced from Lemma \ref{lem1}.
\begin{lem} \label{lem2}
Let $a, b, c, d>0$ and $D_i(i=1,2,\cdots,7)$ be defined by $(\ref{D})$.
\begin{itemize}
\item[(i)] If $(a,b,c,d)\in D_1$, then the sequence
    \begin{equation} \label{r-gam-3}
    k \mapsto \frac{\Gamma(ck+d)}{\Gamma(ak+b)}
    \end{equation}
    is increasing (decreasing) for all $k\geq0$ if and only if $d \geq (\leq) b$;
\item[(ii)] if $(a,b,c,d)\in D_2 \cup D_3$, then the sequence (\ref{r-gam-3}) is decreasing for all $k\geq0$;
\item[(iii)] if $(a,b,c,d)\in D_4$, then there exist an integer $m\geq1$ such that the sequence (\ref{r-gam-3}) is increasing for all $0\leq k \leq m$ and decreasing for all $m \leq k$;
\item[(iv)] if $(a,b,c,d)\in D_5 \cup D_6$, then the sequence (\ref{r-gam-3}) is increasing for all $k\geq0$;
\item[(v)] if $(a,b,c,d)\in D_7$, then there exist an integer $m\geq1$ such that the sequence (\ref{r-gam-3}) is decreasing for all $0\leq k \leq m$ and decreasing for all $m \leq k$.
\end{itemize}
\end{lem}

Now we consider the monotonicity of the functions (\ref{R1}) and (\ref{R2}), where $\{w_k(t)\}_{k\geq0} \in \mathcal{DW}_1(r) \cup  \mathcal{DW}_2(r) \cup \mathcal{DW}_3(r)$ and $w(t,x) \in \mathcal{CW}_1(0,\infty) \cup  \mathcal{CW}_2(0,\infty) \cup \mathcal{CW}_3(0,\infty)$.

\begin{thm} \label{thm1-1}
Let $a, b, c, d>0$, the functions $\mathcal{DE}_{a,b}$ and $\mathcal{DE}_{c,d}$ converge on $(0,r)$, as well as $D_i(i=1,2,\cdots,7)$ be defined by $(\ref{D})$.
\begin{itemize}
\item[(i)] If $(a,b,c,d)\in D_1$ and $\{w_k(t)\}_{k\geq0} \in \mathcal{DW}_{11}(r)$, then the function (\ref{R1}) is increasing (decreasing) on $(0,r)$ if and only if $d \geq (\leq) b$;
\item[(ii)] if $(a,b,c,d)\in D_1$ and $\{w_k(t)\}_{k\geq0} \in \mathcal{DW}_{12}(r)$, then the function (\ref{R1}) is decreasing (increasing) on $(0,r)$ if and only if $d \geq (\leq) b$;
\item[(iii)] if $(a,b,c,d)\in D_2 \cup D_3$ and $\{w_k(t)\}_{k\geq0} \in \mathcal{DW}_{11}(r) (\mathcal{DW}_{12}(r))$, then the function (\ref{R1}) is decreasing (increasing) on $(0,r)$;
\item[(iv)] if $(a,b,c,d)\in D_4$ and $\{w_k(t)\}_{k\geq0} \in \mathcal{DW}_{2}(r)$, then the function (\ref{R1}) is increasing on $(0,r)$ if and only if $H_{\mathcal{DE}_{a,b},\mathcal{DE}_{c,d}}(r^-)\geq 0$ and there exists $t_0\in (0,r)$ such that the function (\ref{R1}) is increasing on $(0,t_0)$ and decreasing on $(t_0,r)$ if $H_{\mathcal{DE}_{a,b},\mathcal{DE}_{c,d}}(r^-) < 0$;
\item[(v)] if $(a,b,c,d)\in D_4$ and $\{w_k(t)\}_{k\geq0} \in \mathcal{DW}_{3}(r)$, then the function (\ref{R1}) is decreasing on $(0,r)$ if and only if $H_{\mathcal{DE}_{a,b},\mathcal{DE}_{c,d}}(0^+)\geq 0$, and there exists $t_0\in (0,r)$ such that the function (\ref{R1}) is increasing on $(0, t_0)$ and decreasing on $(t_0,r)$ if $H_{\mathcal{DE}_{a,b},\mathcal{DE}_{c,d}}(0^+)< 0$;
\item[(vi)] if $(a,b,c,d)\in D_5 \cup D_6$ and $\{w_k(t)\}_{k\geq0} \in \mathcal{DW}_{11}(r) (\mathcal{DW}_{12}(r))$, then the function (\ref{R1}) is increasing (decreasing) on $(0,r)$;
\item[(vii)] if $(a,b,c,d)\in D_7$ and $\{w_k(t)\}_{k\geq0} \in \mathcal{DW}_{2}(r)$, then the function (\ref{R1}) is decreasing on $(0,r)$ if and only if $H_{\mathcal{DE}_{a,b},\mathcal{DE}_{c,d}}(r^-)\leq 0$ and there exists $t_0\in (0,r)$ such that the function (\ref{R1}) is decreasing on $(0,t_0)$ and increasing on $(t_0,r)$ if $H_{\mathcal{DE}_{a,b},\mathcal{DE}_{c,d}}(r^-) > 0$.
\item[(viii)] if $(a,b,c,d)\in D_7$ and $\{w_k(t)\}_{k\geq0}\in \mathcal{DW}_{3}(r)$, then the function (\ref{R1}) is increasing on $(0,r)$ if and only if $H_{\mathcal{DE}_{a,b},\mathcal{DE}_{c,d}}(0^+)\leq 0$, and there exists $t_0\in (0,r)$ such that the function (\ref{R1}) is decreasing on $(0, t_0)$ and increasing on $(t_0,r)$ if $H_{\mathcal{DE}_{a,b},\mathcal{DE}_{c,d}}(0^+)> 0$.
\end{itemize}
\end{thm}

\begin{rem}
This theorem can be obtained by Monotonicity rules \ref{mr1-0}, \ref{mr1-1}, and \ref{mr1-2}, as well as Lemma \ref{lem2}.
\end{rem}

Note that $\{w_k(t)\}_{k\geq0}=\{t^k\}_{k\geq0} \in \mathcal{DW}_{11}(r) \cap \mathcal{DW}_{2}(r)$. Taking $w_k(t)=t^k$,
the function (\ref{R1}) changes into
\begin{equation} \label{R-E}
t \mapsto \frac{\mathbb{E}_{a,b}(t)}{\mathbb{E}_{c,d}(t)} = \frac{\sum_{k=0}^\infty \frac{t^k}{\Gamma(a k+b)}}{\sum_{k=0}^\infty \frac{t^k}{\Gamma(c k+d)}},
\end{equation}
and Theorem \ref{thm1-1} reduces to the following corollary, where $\mathbb{E}_{a,b}(t)$ is the Mittag-Leffler function,  which is convergent for all $t\in(0,\infty)$ if $a,b>0$.
\begin{cor}
Let $a,b,c,d>0$ and $D_i(i=1,2,\cdots,7)$ be defined by $(\ref{D})$.
\begin{itemize}
\item[(i)] If $(a,b,c,d)\in D_1$, then the function (\ref{R-E}) is increasing (decreasing) on $(0,\infty)$ if and only if $d \geq (\leq) b$;
\item[(ii)] if $(a,b,c,d)\in D_2 \cup D_3$, then the function (\ref{R-E}) is decreasing on $(0,\infty)$;
\item[(iii)] if $(a,b,c,d)\in D_4$, then there exists $t_0\in (0,\infty)$ such that the function (\ref{R-E}) is increasing on $(0,t_0)$ and decreasing on $(t_0,\infty)$;
\item[(iv)] if $(a,b,c,d)\in D_5 \cup D_6$, then the function (\ref{R-E}) is increasing on $(0,\infty)$;
\item[(v)] if $(a,b,c,d)\in D_7$, then there exists $t_0\in (0,\infty)$ such that the function (\ref{R-E}) is decreasing on $(0,t_0)$ and increasing on $(t_0,\infty)$.
\end{itemize}
\end{cor}

The following theorem provides the monotonicity of the function (\ref{R2}).
\begin{thm} \label{thm1-2}
Let $a, b, c, d>0$, the functions $\mathcal{CE}_{a,b}$ and $\mathcal{CE}_{c,d}$ converge for all $x\in(0,\infty)$, as well as $D_i(i=1,2,\cdots,7)$ be defined by $(\ref{D})$.
\begin{itemize}
\item[(i)] If $(a,b,c,d)\in D_1$ and $w(t,x) \in \mathcal{CW}_{11}(0,\infty)$, then the function (\ref{R2}) is increasing (decreasing) on $(0,\infty)$ if and only if $d \geq (\leq) b$;
\item[(ii)] if $(a,b,c,d)\in D_1$ and $w(t,x) \in \mathcal{CW}_{12}(0,\infty)$, then the function (\ref{R2}) is decreasing (increasing) on $(0,\infty)$ if and only if $d \geq (\leq) b$;
\item[(iii)] if $(a,b,c,d)\in D_2 \cup D_3$ and $w(t,x) \in \mathcal{CW}_{11}(0,\infty) (\mathcal{CW}_{12}(0,\infty))$, then the function (\ref{R2}) is decreasing (increasing) on $(0,\infty)$;
\item[(iv)] if $(a,b,c,d)\in D_4$ and $w(t,x) \in \mathcal{CW}_{2}(0,\infty)$, then the function (\ref{R2}) is increasing on $(0,\infty)$ if and only if $H_{\mathcal{CE}_{a,b},\mathcal{CE}_{c,d}}(\infty)\geq 0$ and there exists $t_0\in (0,\infty)$ such that the function (\ref{R2}) is increasing on $(0,t_0)$ and decreasing on $(t_0,\infty)$ if $H_{\mathcal{CE}_{a,b},\mathcal{CE}_{c,d}}(\infty) < 0$;
\item[(v)] if $(a,b,c,d)\in D_4$ and $w(t,x) \in \mathcal{CW}_{3}(0,\infty)$, then the function (\ref{R2}) is decreasing on $(0,\infty)$ if and only if $H_{\mathcal{CE}_{a,b},\mathcal{CE}_{c,d}}(0^+)\geq 0$, and there exists $t_0\in (0,\infty)$ such that the function (\ref{R2}) is increasing on $(0, t_0)$ and decreasing on $(t_0,\infty)$ if $H_{\mathcal{CE}_{a,b},\mathcal{CE}_{c,d}}(0^+)< 0$;
\item[(vi)] if $(a,b,c,d)\in D_5 \cup D_6$ and $w(t,x) \in \mathcal{CW}_{11}(0,\infty) (\mathcal{DW}_{12}(0,\infty))$, then the function (\ref{R2}) is increasing (decreasing) on $(0,\infty)$;
\item[(vii)] if $(a,b,c,d)\in D_7$ and $w(t,x) \in \mathcal{CW}_{2}(0,\infty)$, then the function (\ref{R2}) is decreasing on $(0,\infty)$ if and only if $H_{\mathcal{CE}_{a,b},\mathcal{CE}_{c,d}}(\infty)\leq 0$ and there exists $t_0\in (0,\infty)$ such that the function (\ref{R2}) is decreasing on $(0,t_0)$ and increasing on $(t_0,\infty)$ if $H_{\mathcal{CE}_{a,b},\mathcal{CE}_{c,d}}(\infty) > 0$.
\item[(viii)] if $(a,b,c,d)\in D_7$ and $w(t,x) \in \mathcal{CW}_{3}(0,\infty)$, then the function (\ref{R2}) is increasing on $(0,\infty)$ if and only if $H_{\mathcal{CE}_{a,b},\mathcal{CE}_{c,d}}(0^+)\leq 0$, and there exists $t_0\in (0,\infty)$ such that the function (\ref{R2}) is decreasing on $(0, t_0)$ and increasing on $(t_0,\infty)$ if $H_{\mathcal{CE}_{a,b},\mathcal{CE}_{c,d}}(0^+)> 0$.
\end{itemize}
\end{thm}

\begin{rem}
This theorem can be obtained by Monotonicity rules \ref{mr2-0}, \ref{mr2-1}, and \ref{mr2-2}, as well as Lemma \ref{lem1}.
\end{rem}

Noting that $w_k(t,x)=e^{-tx} \in \mathcal{CW}_{12}(0,\infty) \cap \mathcal{CW}_{3}(0,\infty)$, the function (\ref{R1}) changes into
\begin{equation} \label{R-La}
t \mapsto \frac{\mathcal{L}\{\frac{1}{\Gamma(at+b)} \}(x)}{\mathcal{L}\{\frac{1}{\Gamma(ct+d)} \}(x)} = \frac{ \int_0^\infty \frac{1}{\Gamma(at+b)} e^{-tx} \textrm{d} t }{\int_0^\infty \frac{1}{\Gamma(ct+d)} e^{-tx} \textrm{d} t},
\end{equation}
and Theorem \ref{thm1-2} reduces to the following corollary.

\begin{cor}
Let $a, b, c, d>0$ and $D_i(i=1,2,\cdots,7)$ be defined by $(\ref{D})$.
\begin{itemize}
\item[(i)] If $(a,b,c,d)\in D_1$, then the function (\ref{R-La}) is decreasing (increasing) on $(0,\infty)$ if and only if $d \geq (\leq) b$;
\item[(ii)] if $(a,b,c,d)\in D_2 \cup D_3$, then the function (\ref{R-La}) is increasing on $(0,\infty)$;
\item[(iii)] if $(a,b,c,d)\in D_4$, then the function (\ref{R-La}) is decreasing on $(0,\infty)$ if and only if
    \begin{equation*}
      H_{\mathcal{L}\{\frac{1}{\Gamma(at+b)} \},\mathcal{L}\{\frac{1}{\Gamma(ct+d)} \}}(0^+)\geq 0,
    \end{equation*}
    and there exists $t_0\in (0,\infty)$ such that the function (\ref{R-La}) is increasing on $(0, t_0)$ and decreasing on $(t_0,\infty)$ if
    \begin{equation*}
      H_{\mathcal{L}\{\frac{1}{\Gamma(at+b)} \},\mathcal{L}\{\frac{1}{\Gamma(ct+d)} \}}(0^+)< 0;
    \end{equation*}
\item[(iv)] if $(a,b,c,d)\in D_5 \cup D_6$, then the function (\ref{R-La}) is decreasing on $(0,\infty)$;
\item[(v)] if $(a,b,c,d)\in D_7$, then the function (\ref{R-La}) is increasing on $(0,\infty)$ if and only if \begin{equation*}
      H_{\mathcal{L}\{\frac{1}{\Gamma(at+b)} \},\mathcal{L}\{\frac{1}{\Gamma(ct+d)} \}}(0^+)\leq 0,
    \end{equation*}
    and there exists $t_0\in (0,\infty)$ such that the function (\ref{R-La}) is decreasing on $(0, t_0)$ and increasing on $(t_0,\infty)$ if
    \begin{equation*}
      H_{\mathcal{L}\{\frac{1}{\Gamma(at+b)} \},\mathcal{L}\{\frac{1}{\Gamma(ct+d)} \}}(0^+)> 0.
    \end{equation*}
\end{itemize}
\end{cor}

\subsection{Applications in stochastic orders}
In probability theory and statistics, stochastic orders establish a framework to systematically compare the relative magnitudes of random variables.

Stochastic orders have significant applications across various fields, such as finance, operations research, reliability theory, and decision-making under uncertainty. For example, in finance \cite{Kijima-MMOR-1999}, stochastic dominance is used to compare the performance of different investment portfolios, helping investors make informed decisions. In reliability theory \cite{Navarro-EJOR-2010}, stochastic orders can be employed to compare the lifetimes of different systems or components.

Let $Z\geq0$ be a random variable with distribution function $F_Z$. Then the Laplace-Stieltjes transform of $Z$ is defined by
\begin{equation*}
L_{Z}(x) = \int_0^\infty e^{-tx} \textrm{d} F_Z(t), \quad x>0.
\end{equation*}

A random variable $X\geq 0$ is said to be smaller than $Y\geq0$ in the Laplace transform ratio order \cite{Shaked-JAP-1997} if the function
\begin{equation*}
x\mapsto \frac{L_{Y}(x)}{L_{X}(x)}
\end{equation*}
is decreasing for all $x>0$, which denote as $X\leq_{Lt-r} Y$.

In this section, we will establish some sufficient conditions for $X\leq_{Lt-r} Y$ when both $X$ and $Y$ are discrete and continuous random variables, respectively.

\begin{thm} \label{thm2-1-1}
Let $P(X=k)=p_k\geq 0$ and $P(Y=k)=q_k\geq0$ for all $k=0,1,2,\cdots$. If one of conditions
\begin{itemize}
\item[(i)] the sequence $p_k/q_k$ is increasing (decreasing) for all $k\geq0$;
\item[(ii)] there exist an integer $m\geq1$ such that the sequence $\{p_k/q_k\}$ is increasing (decreasing) for all $0\leq k \leq m$ and decreasing (increasing) for all $k \geq m$ and $H_{L_{X},L_{Y}}(0^+) \geq (\leq) 0$;
\end{itemize}
holds, then $Y\leq_{Lt-r} X (X\leq_{Lt-r} Y)$.
\end{thm}

\begin{proof}
By using Corollaries \ref{cor-DLap-1} and \ref{cor-DLap-2}, we obtain that the function
\begin{equation*}
x \mapsto \frac{L_{X}(x)}{L_{Y}(x)} = \frac{\sum_{k=0}^\infty p_k e^{-k x} }{\sum_{k=0}^\infty q_k e^{-k x}}
\end{equation*}
is decreasing (increasing) on $(0,\infty)$, namely, $Y\leq_{Lt-r} X (X\leq_{Lt-r} Y)$.
\end{proof}

\begin{thm} \label{thm2-1-2}
Let the probability density functions of $X$ and $Y$ be $f_X\geq 0$ and $f_Y\geq0$, respectively. If one of conditions
\begin{itemize}
\item[(i)] the function $f_X/f_Y$ is increasing (decreasing) on $(0,\infty)$;
\item[(ii)] there exist $t^*\in(0,\infty)$ such that the function $t\mapsto f_X(t)/f_Y(t)$ is increasing (decreasing) on $(0,t^*)$ and decreasing (increasing) on $(t^*,\infty)$ and $H_{L_{X},L_{Y}}(0^+) \geq (\leq) 0$;
\end{itemize}
holds, then $Y\leq_{Lt-r} X (X\leq_{Lt-r} Y)$.
\end{thm}

\begin{proof}
By using Corollaries \ref{cor-lap-1} and \ref{cor-lap-2}, we obtain that the function
\begin{equation*}
x \mapsto \frac{L_{X}(x)}{L_{Y}(x)} = \frac{ \int_0^\infty f_X(t) e^{-tx} \textrm{d} t }{\int_0^\infty f_Y(t) e^{-tx} \textrm{d} t}
\end{equation*}
is decreasing (increasing) on $(0,\infty)$, namely, $Y\leq_{Lt-r} X (X\leq_{Lt-r} Y)$.
\end{proof}


\end{document}